\documentclass{article}
\usepackage{geometry}
\usepackage{amsmath,amssymb,amsfonts}
\usepackage{amsthm}
\usepackage{mathrsfs}
\usepackage[title]{appendix}
\usepackage{textcomp}
\usepackage{manyfoot}

\newtheorem{theorem}{Theorem}
\newtheorem{remark}{Remark}%
\newtheorem{lemma}{Lemma}%

\newtheorem{definition}{Definition}%
\newtheorem{ left=1.5cm,right=1.5cm, top=2cm, bottom=2cm}{geometry}
\usepackage{authblk}
\title{Precompactness in bivariate metric
	semigroup-valued bounded variation spaces}
\author[1]{Yinglian Niu}
\author[1,2,3]{Jingshi Xu\textsuperscript{*}}
\affil[1]{School of Mathematics and Computing Science,Guilin University of Electronic Technology,1 Jinjilu,Guilin,541004, Guangxi,China}
\affil[2]{Center for Applied Mathematics of Guangxi,Guilin University of Electronic Technology,1 Jinjilu,Guilin,541004, Guangxi, China}
\affil[3]{Guangxi Colleges and Universities Key Laboratory of Data Analysis and Computation, Guilin University of Electronic Technology, 1 Jinjilu,Guilin,541004,Guangxi,China}
\date{}

\begin{document}

\maketitle
\vspace{-1\baselineskip} 
\begin{center}
	\*Corresponding author(s). E-mail(s): jingshixu@126.com;
	
	\ Contributing authors: 1516035825@qq.com;
	
	\ These authors contributed equally to this work.
\end{center}

\section*{\centering{\small{Abstract}}}
\vspace{-0.5\baselineskip}
In this paper, we show that if a set in bivariate metric semigroups-valued bounded variation spaces is pointwise totally bounded  and joint equivariated then it is precompact. These spaces include bounded Jordan variation spaces, bounded Wiener variation spaces, bounded Waterman variation spaces, bounded Riesz variation spaces and bounded Korenblum variation spaces. To do so, we introduce the concept of equimetric set.\\
\textbf{Keywords:} bounded variation space; precompact set; metric semigroup; joint equivariated; bivariate; equimetric set

\maketitle

\section{Introduction}
The concept of bounded variation was first introduced by Jordan in his study of Fourier series in \cite{20}. Subsequently, the bounded variation space is extended to a wider range of directions, many bounded variation spaces and their applications are proposed; see \cite{1,2,3,bbc,5,7,8,9,10,11,12,14,13,15,18,24,26} and references therein.

It is well known that precompact sets are important in topological spaces. For a long time, there are a few results for the judgment of precompact sets in bounded variation spaces; see \cite{16, 19, 23}. Recently, in \cite{6} Bugajewski and Gulgowski introduced the concept of equivariation set and gave the criterion of precompactness for subsets in univariate bounded variation space in the sense of Jordan. The results in \cite{6} were generalized by Si and the second author of this paper to univariate Banach space-valued bounded-variation spaces in the sense of Jordan, bounded Wiener variation spaces, bounded Wiener-Young variation spaces, bounded Schramm variation spaces, bounded Waterman variation spaces, bounded Riesz variation spaces and bounded Korenblum variation spaces in \cite{25}, and variable exponent bounded variation spaces in \cite{sx}. They also were generalized to bounded Waterman $\Lambda$-variation spaces, Young $\Phi$-variation spaces and integral variation spaces by in Gulgowski\cite{30}, and bounded Schramm variation spaces by Gulgowski, Kasprzak and Ma\'{c}kowiak in \cite{31} via equinormed sets. Recently, the authors of the paper generalized the results in \cite{25} to univariate metric semigroup valued bounded-variation spaces in \cite{nx}. We remark here that the theory of metric semigroup valued bounded-variation spaces has been extended over the past decades, see \cite{15}.

Inspired by the above works, we consider the for sufficient conditions for precompact sets in bivariate metric semigroup-valued bounded variation spaces. The structure of this paper is as follows. In Sections 2, we obtain a sufficient condition for precompact sets in bivariate metric semigroup valued bounded Jordan variation spaces. In Sections 3 and 4, we consider the sufficient conditions for the precompact sets in bivariate metric semigroup valued bounded Wiener variation spaces in two cases respectively. From Section 5 to Section 7, we give the sufficient conditions for precompact sets in bivariate metric semigroup valued bounded Watman variation spaces, bounded Riesz variation spaces, bounded Korenblum variation spaces respectively.

In the remainder of this section, we recall some notions. In the sequel, by $\mathbb{N}$ we denote the set of all positive integers, $\mathbb{R}$ represents the set of real numbers, $I$ represents the unit interval $[0,1]$, then $I\times I=[0,1]\times[0,1]$. Take a pair of partitions $\Pi:0=t_{0}<t_{1}<\cdots<t_{n}=1$ and $\Pi^{\ast}:0=s_{0}<s_{1}<\cdots<s_{m}=1$ of the unit interval $I$. A metric semigroup is a triple $(M,d,+)$, where $(M,d)$ is a metric space with metric $d$, while $(M,+)$ is an additive commutative semigroup with addition operation, and $d$ is translation invariant: $d(u+w,v+w)=d(u,v)$ for all $u,v,w\in M$. A metric semigroup $(M,d,+)$ is complete if $(M,d)$ is a complete metric space. In this paper, we always assume the metric semigroup $(M,d,+)$ is complete. For arbitrary elements $u,v,\bar{u},\bar{v}\in M$ of a metric semigroup $(M,d,+)$, we have
\begin{equation}\label{1}
 d(u,v)\leq d(u+\bar{u},v+\bar{v})+d(\bar{u},\bar{v}).
\end{equation}
\begin{equation}\label{2}
 d(u+\bar{u},v+\bar{v})\leq d(u,v)+d(\bar{u},\bar{v}).
\end{equation}

A subset in a topological space is precompact if its closure is compact. We say $\rho$ is a pseudo-metric on $X$ if $\rho(f,g)=\rho(g,f)\geq 0$ and $\rho(f,g)\leq \rho(f,h)+\rho(h,g)$ for each $f,g,h\in X$. A subset $M$ is totally bounded in a pseudo-metric space $X$ if for each $\varepsilon>0$ there exist finite elements $\{x_1,x_2,x_3,...,x_n\}$ in $M$ such that $M\subset\bigcup^n_{i=1}B(x_i,\varepsilon)$, where $B(x_i,\varepsilon)$ is a ball in $X$ with center at $x_{i}$ and radius $\varepsilon$, then $\{B(x_i,\varepsilon),i=1,2,3,...,n\}$ is called the finite $\varepsilon$-net of $M$.
In a complete metric space, precompactness is equivalent to totally boundedness set.

Suppose $\{\rho_\tau: \tau\in T\}$ is a family of pseudo-metrics on $X$. Let $\rho=\sup_{\tau\in T}\rho_\tau$. That means for each $f,g\in X,$
\[\rho(f,g)=\sup_{\tau\in T}\rho_\tau(f,g).\]
Then $\rho$ is also a pseudo-metric.

Let $A\subset X$. If for any $\varepsilon>0,$ there exists a $\tau\in T$ such that
\[\rho(f,g)<\varepsilon+\rho_{\tau}(f,g), \ \text{for every} \ f,g\in A,\]
then $A$ is called $T$-equimetric.

We remark here that the concept of equimetric sets is inspired by the concept of equinormed sets in \cite{30}.

It is easy to obtain the following lemma.

\begin{lemma}\label{l1}
Suppose $\{\rho_\tau: \tau\in T\}$ is a family of pseudo-metrics on $X$ and $\rho=\sup_{\tau\in T}\rho_\tau$. Let $A\subset X$. If $A$ is $T$-equimetric and totally bounded for $\rho_\tau$ for each $\tau \in T$, then $A$ is totally bounded for $\rho$.
\end{lemma}

By use of Lemma 1, we will obtain simpler proofs than those in \cite{nx} for univariate setting in the following sections.

\section{Bivariate bounded Wiener $p$ variation spaces as $p\in(0,\infty)$}

Norbert Wiener introduced functions of bounded $p$-variation and used them to study Fourier series in \cite{27}. In this section, we consider them for bivariate metric semigroup valued functions. For a function $f:I\times I\rightarrow M$, and a pair of partitions $\Pi$ and $\Pi^{\ast}$ of the unit interval $I$, if $p\in [1,\infty )$ we denote
$$V_{p}(f(\cdot,0),\Pi):= \bigg(\sum_{i=1}^{n}d^{p}(f(t_{i},0),f(t_{i-1},0))\bigg)^{\frac{1}{p}},$$
$$V_{p}(f(0,\cdot),\Pi^{\ast}):=\bigg(\sum_{j=1}^{m}d^{p}(f(0,s_{j}),f(0,s_{j-1}))\bigg)^{\frac{1}{p}},$$
$$V_{p,2}(f,\Pi\times\Pi^{\ast}):=\bigg(\sum_{j=1}^{m}\sum_{i=1}^{n}d^{p}(f(t_{i},s_{j})+f(t_{i-1},s_{j-1}),f(t_{i},s_{j-1})+f(t_{i-1},s_{j}))
\bigg)^{\frac{1}{p}};$$
if $p\in(0,1)$, we denote
$$V_{p}(f(\cdot,0),\Pi):= \sum_{i=1}^{n}d^{p}(f(t_{i},0),f(t_{i-1},0)),$$
$$V_{p}(f(0,\cdot),\Pi^{\ast}):=\sum_{j=1}^{m}d^{p}(f(0,s_{j}),f(0,s_{j-1})),$$
and
$$V_{p,2}(f,\Pi\times\Pi^{\ast}):=\sum_{j=1}^{m}\sum_{i=1}^{n}d^{p}(f(t_{i},s_{j})+f(t_{i-1},s_{j-1}),f(t_{i},s_{j-1})+f(t_{i-1},s_{j}).$$
Moreover, $V_{p}(f(\cdot,0)):= \sup V_{p}(f(\cdot,0),\Pi)$,~$V_{p}(f(0,\cdot)):= \sup V_{p}(f(0,\cdot),\Pi^{\ast})$,~$V_{p,2}(f):= \sup V_{p,2}(f,\Pi\times\Pi^{\ast})$, where all suprema are taken over the indicated partitions.

\begin{definition}\label{d1}
 Let $p\in(0,\infty )$. For a bounded function $f:I\times I\to M$, the quantity
$$V_{p}(f):= V_{p}(f(\cdot,0))+V_{p}(f(0,\cdot))+V_{p,2}(f)$$
is called the total p-variation of $f$ on $I\times I$. If $V(f)<+\infty$, we say that $f$ has bounded p-variation on $I\times I$ and denote ${\rm WBV}_{p}(I\times I,M)=\{f:I\times I\rightarrow M,V_{p}(f)<\infty\}$.
\end{definition}

\begin{remark}\label{r1}
 $1$-variation is also called Jordan variation.
\end{remark}

Taking a pair of partitions $\Pi$ and $\Pi^{\ast}$ of the unit interval $I$, for every $f,g\in{\rm WBV}_{p}(I \times I,M)$, if $p\in [1,\infty )$, we denote
$$V_{p}(f(\cdot,0),g(\cdot,0),\Pi):= \bigg(\sum_{i=1}^{n}d^{p}(f(t_{i},0)+g(t_{i-1},0),f(t_{i-1},0)+g(t_{i},0))\bigg)^{\frac{1}{p}},$$
$$V_{p}(f(0,\cdot),g(0,\cdot),\Pi^{\ast}):=\bigg(\sum_{j=1}^{m}d^{p}(f(0,s_{j})+g(0,s_{j-1}),f(0,s_{j-1})+g(0,s_{j}))\bigg)^{\frac{1}{p}},$$ and
\begin{align*}
V_{p,2}(f,g,\Pi\times\Pi^{\ast}):=\bigg(\sum_{j=1}^{m}\sum_{i=1}^{n}&d^{p}(f(t_{i},s_{j})+f(t_{i-1},s_{j-1})+g(t_{i},s_{j-1})+g(t_{i-1},s_{j}),\\
&g(t_{i},s_{j})+g(t_{i-1},s_{j-1})+f(t_{i},s_{j-1})+f(t_{i-1},s_{j}))\bigg)^{\frac{1}{p}};
\end{align*}
if $p\in(0,1)$, we denote
$$V_{p}(f(\cdot,0),g(\cdot,0),\Pi):= \sum_{i=1}^{n}d^{p}(f(t_{i},0)+g(t_{i-1},0),f(t_{i-1},0)+g(t_{i},0)),$$
$$V_{p}(f(0,\cdot),g(0,\cdot),\Pi^{\ast}):=\sum_{j=1}^{m}d^{p}(f(0,s_{j})+g(0,s_{j-1}),f(0,s_{j-1})+g(0,s_{j})),$$ and
\begin{align*}
V_{p,2}(f,g,\Pi\times\Pi^{\ast}):=\sum_{j=1}^{m}\sum_{i=1}^{n}d^{p}(&f(t_{i},s_{j})+f(t_{i-1},s_{j-1})+g(t_{i},s_{j-1})+g(t_{i-1},s_{j}),\\
&g(t_{i},s_{j})+g(t_{i-1},s_{j-1})+f(t_{i},s_{j-1})+f(t_{i-1},s_{j})).
\end{align*}

\begin{definition}\label{d2}
 Let $p\in(0,\infty )$. For every $f,g\in{\rm WBV}_{p}(I \times I,M)$ and each pair of partitions $\Pi$ and $\Pi^{\ast}$, we denote
\begin{align*}
V^{p}_{\Pi\times\Pi^{\ast}}(f,g)=V_{p}(f(\cdot,0),g(\cdot,0),\Pi)+V_{p}(f(0,\cdot),g(0,\cdot),\Pi^{\ast})+V_{p,2}(f,g,\Pi\times\Pi^{\ast}).
\end{align*}
Then the joint variation of two functions $f$, $g$ is defined by $$V_{p}(f,g):= \sup_{\Pi\times\Pi^\ast} V^{p}_{\Pi\times\Pi^\ast}(f,g)$$
 where the supremum is taken over the indicated partitions.  For every $f,g\in{\rm WBV}_{p}(I \times I,M)$, we denote
\begin{align*}\rho_{p}(f,g):= \sup_{\Pi\times\Pi^\ast}\rho^{p}_{\Pi\times\Pi^\ast}(f,g)&=:\sup_{\Pi\times\Pi^\ast} (d(f(0,0),g(0,0))+V^{p}_{\Pi\times\Pi^\ast}(f,g))\\
&= d(f(0,0),g(0,0))+V_{p}(f,g) .\end{align*}
\end{definition}

\begin{lemma}\label{l2}
Let $p\in(0,\infty )$. If $f,g\in{\rm WBV}_{p}(I\times I,M)$, then $V_{p}(f,g)\leq V_{p}(f)+V_{p}(g)$ .
\end{lemma}

 \begin{proof} By Definition \ref{d2}, $$V_{p}(f,g)=V_{p}(f(\cdot,0),g(\cdot,0))+V_{p}(f(0,\cdot),g(0,\cdot))+V_{p,2}(f,g).$$
We claim that $V_{p,2}(f,g)\leq V_{p,2}(f)+V_{p,2}(g)$. Below, we only consider the case $p\in [1,\infty)$. The argument for $p\in (0,1)$ is similar. In fact, by Minkowski's inequality
\begin{align*}
V_{p,2}(f,g,\Pi\times\Pi^{\ast})&=\bigg(\sum_{j=1}^{m}\sum_{i=1}^{n}d^{p}(f(t_{i},s_{j})+f(t_{i-1},s_{j-1})+g(t_{i},s_{j-1})+g(t_{i-1},s_{j}),\\
&~~~~~~~~~~~~~g(t_{i},s_{j})+g(t_{i-1},s_{j-1})+f(t_{i},s_{j-1})+f(t_{i-1},s_{j}))\bigg)^{\frac{1}{p}}\\
&\leq\bigg(\sum_{j=1}^{m}\sum_{i=1}^{n}d^{p}(f(t_{i},s_{j})+f(t_{i-1},s_{j-1}),f(t_{i},s_{j-1})+f(t_{i-1},s_{j}))\bigg)^{\frac{1}{p}}\\
&~~+\bigg(\sum_{j=1}^{m}\sum_{i=1}^{n}d^{p}(g(t_{i},s_{j})+g(t_{i-1},s_{j-1}),g(t_{i},s_{j-1})+g(t_{i-1},s_{j}))\bigg)^{\frac{1}{p}}\\
&\leq V_{p,2}(f)+V_{p,2}(g).
\end{align*}
Taking the upper bound of $\Pi$ and $\Pi^{\ast}$ to the left of the upper inequality, we have
\begin{align*}
V_{p,2}(f,g)\leq V_{p,2}(f)+V_{p,2}(g).
\end{align*}
Similarly, $V_{p}(f(\cdot,0),g(\cdot,0))\leq V_{p}(f(\cdot,0))+V_{p}(g(\cdot,0))$ and $V_{p}(f(0,\cdot),g(0,\cdot))\leq V_{p}(f(0,\cdot))+V_{p}((0,\cdot))$. Therefore, $V_{p}(f,g)\leq V_{p}(f)+V_{p}(g)$.
\end{proof}

\begin{lemma}\label{l3}
Let $p\in(0,\infty )$. If $f,g\in{\rm WBV}_{p}(I\times I,M)$, then $V_{p}(f)\leq V_{p}(f,g)+V_{p}(g)$ and $|V_{p}(f)-V_{p}(g)|\leq V_{p}(f,g)$.
\end{lemma}
\begin{proof}
For every partitions of $\Pi$ and $\Pi^{\ast}$. By (\ref{1}),
\begin{align*}
d(f(t_{i},0),f(t_{i-1},0))\leq d(f(t_{i},0)+g(t_{i-1},0),f(t_{i-1},0)+g(t_{i},0))+d(g(t_{i},0),g(t_{i-1},0)),
\end{align*}
\begin{align*}
d(f(0,s_{j}),f(0,s_{j-1}))\leq d(f(0,s_{j})+g(0,s_{j-1}),f(0,s_{j-1})+g(0,s_{j}))+d(g(0,s_{j}),g(0,s_{j-1})),
\end{align*}
and
\begin{align*}
&d(f(t_{i},s_{j})+f(t_{i-1},s_{j-1}),f(t_{i},s_{j-1})+f(t_{i-1},s_{j}))\\
&\leq d(f(t_{i},s_{j})+f(t_{i-1},s_{j-1})+g(t_{i},s_{j-1})+g(t_{i-1},s_{j}),g(t_{i},s_{j})+g(t_{i-1},s_{j-1})\\
&~~~~+f(t_{i},s_{j-1})+f(t_{i-1},s_{j}))+d(g(t_{i},s_{j})+g(t_{i-1},s_{j-1})+g(t_{i},s_{j-1})+g(t_{i-1},s_{j})).
\end{align*}
Below, we only consider the case $p\in [1,\infty)$. The argument for $p\in (0,1)$ is similar.  By Minkowski's inequality, we have
\begin{align*}
V_{p}(f(\cdot,0),\Pi)&=\bigg(\sum_{i=1}^{n}d^{p}(f(t_{i},0),f(t_{i-1},0))\bigg)^{\frac{1}{p}}\\
&\leq V_{p}(f(\cdot,0),g(\cdot,0),\Pi)+V_{p}(g(\cdot,0),\Pi),
\end{align*}
\begin{align*}
V_{p}(f(0,\cdot),\Pi^{\ast})&=\bigg(\sum_{j=1}^{m}d^{p}(f(0,s_{j}),f(0,s_{j-1})\bigg)^{\frac{1}{p}}\\
&\leq V_{p}(f(0,\cdot),g(0,\cdot),\Pi^{\ast})+V_{p}(g(0,\cdot),\Pi^{\ast}),
\end{align*}
and
\begin{align*}
V_{p,2}(f,\Pi\times\Pi^{\ast})&=\bigg(\sum_{j=1}^{m}\sum_{i=1}^{n}d^{p}(f(t_{i},s_{j})+f(t_{i-1},s_{j-1}),f(t_{i},s_{j-1})+f(t_{i-1},s_{j}))\bigg)^{\frac{1}{p}}\\
&\leq V_{p,2}(f,g,\Pi\times\Pi^{\ast})+V_{p,2}(g,\Pi\times\Pi^{\ast}).
\end{align*}
Take the upper bound of $\Pi$ and $\Pi^{\ast}$ to the left of the upper inequality, we have
\begin{align*}
V_{p}(f(\cdot,0))\leq V_{p}(f(\cdot,0),g(\cdot,0))+V_{p}(g(\cdot,0)),
\end{align*}
\begin{align*}
V_{p}(f(0,\cdot))\leq V_{p}(f(0,\cdot),g(0,\cdot))+V_{p}(g(0,\cdot)),
\end{align*}
and
\begin{align*}
V_{p,2}(f)\leq V_{p,2}(f,g)+V_{p,2}(g).
\end{align*}
Therefore,
\begin{align*}
V_{p}(f)\leq V_{p}(f,g)+V_{p}(g).
\end{align*}
Similarly, we have
\begin{align*}
V_{p}(g)\leq V_{p}(f,g)+V_{p}(f).
\end{align*}
Therefore, $|V_{p}(f)-V_{p}(g)|\leq V_{p}(f,g)$.
\end{proof}

It is easy to see that $\rho^{p}_{\Pi\times\Pi^\ast}$ is a pseudo-metric on ${\rm WBV}(I\times I,M)$ for each pair of partitions $\Pi$ and $\Pi^{\ast}$.

\begin{theorem}\label{t1}
  Let $p\in(0,\infty )$. Then $({\rm WBV}_{p}(I\times I,M),\rho_{p})$ is a complete metric space.
\end{theorem}
\begin{proof} For $p=1$, the result has proved Lemma 3 in \cite{11}. Hence, we only give the proof for $p\in(0,1)$ or $p\in(1,\infty )$. The proof for $p\in(0,1)$ is similar to the case $p\in(1,\infty )$. Thus, we only consuder $p\in(1,\infty )$. First, we prove $\rho_{p}$ is a metric. Let every $f,g,h\in {\rm WBV}_p(I\times I,M)$. Obviously, $\rho_{p}(f,g)\geq 0$. If $\rho_{p}(f,g)=0$, then $d(f(0,0),g(0,0))+V_{p}(f,g)=0$. Therefore $f(0,0)=g(0,0)$. For every $t,s>0$, takeing special partitions $0=t_{0}<t<t_{1}=1$ and $0=s_{0}<s<s_{1}=1$, we have
 $$(d^{p}(f(t,0)+g(0,0),f(0,0)+g(t,0))+d^{p}(f(1,0)+g(t,0),f(t,0)+g(1,0)))^{\frac{1}{p}}\leq V_{p}(f,g)=0,$$
 $$(d^{p}(f(0,s)+g(0,0),f(0,0)+g(0,s))+d^{p}(f(0,1)+g(0,s),f(0,s)+g(0,1)))^{\frac{1}{p}}\leq V_{p}(f,g)=0,$$
 \begin{align*}
 &(d^{p}(f(t,s)+f(0,0)+g(t,0)+g(0,s),f(t,0)+f(0,s)+g(t,s)+g(0,0))\\
 &+d^{p}(f(1,1)+f(t,s)+g(t,1)+g(1,s),f(t,1)+f(1,s)+g(t,s)+g(1,1)))^{\frac{1}{p}}\leq V_{p}(f,g)=0,
 \end{align*}
 which yields $f(t,0)=g(t,0)$,~$f(0,s)=g(0,s)$,~$f(t,s)=g(t,s)$. Thus for every $t,s\in I$, we obtain $f(t,s)=g(t,s)$. Again,
 \begin{align*}
  &V_{p}(f(\cdot,0),g(\cdot,0),\Pi)\\
  &=\bigg(\sum_{i=1}^{n}d^{p}(f(t_{i},0)+g(t_{i-1},0)+h(t_{i},0)+h(t_{i-1},0),f(t_{i-1},0)+g(t_{i},0)+h(t_{i},0)+h(t_{i-1},0))\bigg)^{\frac{1}{p}}\\
  &\leq\bigg(\sum_{i=1}^{n}(d(f(t_{i},0)+h(t_{i-1},0),f(t_{i-1},0)+h(t_{i},0))+d(g(t_{i-1},0)+h(t_{i},0),g(t_{i},0)+h(t_{i-1},0)))^{p}\bigg)^{\frac{1}{p}}\\
  &\leq\bigg(\sum_{i=1}^{n}d^{p}(f(t_{i},0)+h(t_{i-1},0),f(t_{i-1},0)+h(t_{i},0))\bigg)^{\frac{1}{p}}\\
  &~~~~+\bigg(\sum_{i=1}^{n}d^{p}(g(t_{i-1},0)+h(t_{i},0),g(t_{i},0)+h(t_{i-1},0))\bigg)^{\frac{1}{p}}\\
  &=V_{p}(f(\cdot,0),h(\cdot,0),\Pi)+V_{p}(h(\cdot,0),g(\cdot,0),\Pi).
 \end{align*}
 Similarly, $V_{p}(f(0,\cdot),g(0,\cdot),\Pi^{\ast})\leq V_{p}(f(0,\cdot),h(0,\cdot),\Pi^{\ast})+V_{p}(h(0,\cdot),g(0,\cdot),\Pi^{\ast})$ and $V_{p,2}(f,g,\Pi\times\Pi^{\ast})\leq V_{p,2}(f,h,\Pi\times\Pi^{\ast})+V_{p,2}(h,g,\Pi\times\Pi^{\ast})$, moreover, $d(f(0,0),g(0,0))\leq d(f(0,0),h(0,0))+ d(h(0,0),g(0,0))$, we have $\rho_{p}(f,g)\leq\rho_{p}(f,h)+\rho_{p}(h,g)$. Therefore, $\rho_{p}$ is a metric.

Now, we prove ${\rm WBV}_{p}(I\times I,M)$ is complete. Let $\{f_{u}\}_u$ is a \rm Cauchy sequence, then for every $\varepsilon>0$ there exists a integer $N\in\mathbb{N}$ such that for all $u,v\in \mathbb{N}$ with $u,v>N$, $\rho_{p}(f_u,f_v)\leq\varepsilon$. Thus $d(f_u(0,0),f_v(0,0))\leq\varepsilon$ and $V_{p}(f_u,f_v)\leq\varepsilon$ for $u,v>N$. Since $M$ is complete, so that the sequence $\{f_{u}(0,0)\}_u$ converges to some element of $M$ which is denoted by $f(0,0)$. Let $v\to\infty$ and $\lim\limits_{v\to\infty}f_v(0,0)=f(0,0)$. For $u>N$, we get $d(f_u(0,0),f(0,0))\leq\varepsilon$. Fixing an arbitrary point $t,s\in(0,1]$, taking special partitions $\{0,t,1\}$ and $\{0,s,1\}$, we have
 \begin{align*}
  &|d(f_u(t,s),f_v(t,s))-d(f_u(0,0),f_v(0,0))|\\
  &\leq d(f_u(t,s)+f_u(0,0),f_v(0,0)+f_v(t,s))\\
  &\leq(d^{p}(f_u(t,s)+f_u(0,0)+f_v(t,0)+f_v(0,s),f_v(0,0)+f_v(t,s)+f_u(t,0)+f_u(0,s)))^{\frac{1}{p}}\\
  &~~+d^{p}(f_v(t,0)+f_v(0,s),f_u(t,0)+f_u(0,s)))^{\frac{1}{p}}\\
  &\leq(d^{p}(f_u(t,s)+f_u(0,0)+f_v(t,0)+f_v(0,s),f_v(0,0)+f_v(t,s)+f_u(t,0)+f_u(0,s)))^{\frac{1}{p}}\\
  &~~+(d^{p}(f_v(t,0)+f_u(0,0),f_u(t,0)+f_v(0,0)))^{\frac{1}{p}}+d^{p}(f_v(0,s)+f_u(0,0),f_u(0,0)+f_v(0,0)))^{\frac{1}{p}}\\
  &\leq V_{p}(f_u,f_v)\leq\varepsilon,
 \end{align*}
then $d(f_u(t,s),f_v(t,s))\leq 2\varepsilon$ for every $u,v>N$. This means the sequence $\{f_{u}(t,s)\}_u$ is a uniformly \rm Cauchy sequence, it has a pointwise limit which is denoted by $f$.

By Lemma \ref{l3}, $|V_{p}(f_{u})-V_{p}(f_{v})|\leq V_{p}(f_{u},f_{v})$. Therefore the sequence $\{V_{p}(f_u)\}$ is bounded, i.e. there exists a constant $E>0$ such that $V_{p}(f_u)\leq E$ for $u\in\mathbb{N}$. For any finite partitions of $\Pi$ and $\Pi^{\ast}$,
 \begin{align*}
 &\bigg(\sum_{i=1}^{n}d^{p}(f_u(t_{i},0),f_{u}(t_{i-1},0))\bigg)^{\frac{1}{p}}+\bigg(\sum_{j=1}^{m}d^{p}(f_{u}(0,s_{j}),f_{u}(0,s_{j-1}))\bigg)^{\frac{1}{p}}\\
 &+\bigg(\sum_{j=1}^{m}\sum_{i=1}^{n}d^{p}(f_{u}(t_{i},s_{j})+f_{u}(t_{i-1},s_{j-1}),f_{u}(t_{i},s_{j-1})+f_{u}(t_{i-1},s_{j}))\bigg)^{\frac{1}{p}}\leq E,
 \end{align*}
 letting $u\to\infty$, we obtain
  \begin{align*}
 &\bigg(\sum_{i=1}^{n}d^{p}(f(t_{i},0),f(t_{i-1},0))\bigg)^{\frac{1}{p}}+\bigg(\sum_{j=1}^{m}d^{p}(f(0,s_{j}),f(0,s_{j-1}))\bigg)^{\frac{1}{p}}\\
 &+\bigg(\sum_{j=1}^{m}\sum_{i=1}^{n}d^{p}(f(t_{i},s_{j})+f(t_{i-1},s_{j-1}),f(t_{i},s_{j-1})+f(t_{i-1},s_{j}))\bigg)^{\frac{1}{p}}\leq E,
 \end{align*}
 which implies that $f\in {\rm WBV}_{p}(I\times I,M)$. Moreover, $V_{p}(f_u,f_v)\leq\varepsilon$ for $u,v>N$. Therefore, for $u,v>N$
  \begin{align*}
 &\bigg(\sum_{i=1}^{n}d^{p}(f_u(t_{i},0)+f_{v}(t_{i-1},0),f(t_{i-1},0)+f_{v}(t_{i},0))\bigg)^{\frac{1}{p}}\\
 &+\bigg(\sum_{j=1}^{m}d^{p}(f_{u}(0,s_{j})+f_{v}(0,s_{j-1}),f_{u}(0,s_{j-1})+f_{v}(0,s_{j}))\bigg)^{\frac{1}{p}}\\
 &+\bigg(\sum_{j=1}^{m}\sum_{i=1}^{n}d^{p}(f_{u}(t_{i},s_{j})+f_{u}(t_{i-1},s_{j-1})+f_{v}(t_{i},s_{j-1})+f_{v}(t_{i-1},s_{j}),\\
&~~~~~~f_{v}(t_{i},s_{j})+f_{v}(t_{i-1},s_{j-1})+f_{u}(t_{i},s_{j-1})+f_{u}(t_{i-1},s_{j}))\bigg)^{\frac{1}{p}}\leq \varepsilon,
 \end{align*}
 letting $v\to\infty$, we have
  \begin{align*}
 &\bigg(\sum_{i=1}^{n}d^{p}(f_u(t_{i},0)+f(t_{i-1},0),f_{u}(t_{i-1},0)+f(t_{i},0))\bigg)^{\frac{1}{p}}\\
 &+\bigg(\sum_{j=1}^{m}d^{p}(f_{u}(0,s_{j})+f(0,s_{j-1}),f_{u}(0,s_{j-1})+f(0,s_{j}))\bigg)^{\frac{1}{p}}\\
 &+\bigg(\sum_{j=1}^{m}\sum_{i=1}^{n}d^{p}(f_{u}(t_{i},s_{j})+f_{u}(t_{i-1},s_{j-1})+f(t_{i},s_{j-1})+f(t_{i-1},s_{j}),\\
&~~~~~~f(t_{i},s_{j})+f(t_{i-1},s_{j-1})+f_{u}(t_{i},s_{j-1})+f_{u}(t_{i-1},s_{j}))\bigg)^{\frac{1}{p}}\leq \varepsilon.
 \end{align*}
Hence, we have $V_{p}(f_u,f)\leq\varepsilon$ for $u>N$. Therefore, $\rho_{p}(f_u,f)\leq 2\varepsilon$. Thus $\{f_{u}\}$ converges to $f$ in ${\rm WBV}_{p}(I\times I,M)$. So $({\rm WBV}_{p}(I\times I,M),\rho_{p})$ is complete.
\end{proof}

\begin{definition}\label{d3}
Let $p\in(0,\infty)$. A set $A\subset{\rm WBV}_{p}(I\times I,M)$ is said to be joint equivariated, if for every $\varepsilon>0$, there exist partitions $\Pi_{\varepsilon}$ and $\Pi_{\varepsilon}^{\ast}$ such that for every $f,g\in A$, we have
$$V_{p}(f,g)\leq\varepsilon+V_{p}(f,g,\Pi_{\varepsilon}\times\Pi_{\varepsilon}^{\ast}),$$ where $$V_{p}(f,g,\Pi_{\varepsilon}\times\Pi_{\varepsilon}^{\ast}):= V_{p}(f(\cdot,0),g(\cdot,0),\Pi_{\varepsilon})+V_{p}(f(0,\cdot),g(0,\cdot),\Pi_{\varepsilon}^{\ast})+V_{p,2}(f,g,\Pi_{\varepsilon}\times\Pi_{\varepsilon}^{\ast}).$$
\end{definition}

\begin{lemma}\label{l4}
Let $m,n\in \mathbb{N}$. For elements $\xi=(\xi_{ij})_{i=0,j=0}^{n,m}$ and $\delta=(\delta_{ij})_{i=0,j=0}^{n,m}\in M^{(n+1)(m+1)}$. If $p\in(0,1)$, let
\begin{align*}
&\rho'_{p}(\xi,\delta)=\sup_{\Pi\times\Pi^\ast}\bigg( d(\xi_{00},\delta_{00})+\sum_{i=1}^{n}d^{p}(\xi_{i\,0}+\delta_{i-1\,0},\xi_{i-1\,0}+\delta_{i\,0})\\
&~~+\sum_{j=1}^{m}d^{p}(\xi_{0\,j}+\delta_{0\,j-1},\xi_{0\,j-1}+\delta_{0\,j})\\
&~~+\sum_{j=1}^{m}\sum_{i=1}^{n}d^{p}(\xi_{i\,j}+\xi_{i-1\,j-1}+\delta_{i-1\,j}+\delta_{i\,j-1},\xi_{i-1\,j}+\xi_{i\,j-1}+\delta_{i\,j}+\delta_{i-1\,j-1})\bigg).
\end{align*}
If $p\in(1,\infty)$, let
\begin{align*}
\rho'_{p}(\xi,&\delta)= \sup_{\Pi\times\Pi^\ast}\bigg( d(\xi_{00},\delta_{00})+\bigg(\sum_{i=1}^{n}d^{p}(\xi_{i\,0}+\delta_{i-1\,0},\xi_{i-1\,0}+\delta_{i\,0})\bigg)^{\frac{1}{p}}\\
&~~+\bigg(\sum_{j=1}^{m}d^{p}(\xi_{0\,j}+\delta_{0\,j-1},\xi_{0\,j-1}+\delta_{0\,j})\bigg)^{\frac{1}{p}}\\
&~~+\bigg(\sum_{j=1}^{m}\sum_{i=1}^{n}d^{p}(\xi_{i\,j}+\xi_{i-1\,j-1}+\delta_{i-1\,j}+\delta_{i\,j-1},\xi_{i-1\,j}+\xi_{i\,j-1}+\delta_{i\,j}+\delta_{i-1\,j-1})\bigg)^{\frac{1}{p}}\bigg).
\end{align*}
Then $(M^{(n+1)(m+1)},\rho'_{p})$ is complete. Moreover, let $B_{ij}$ be precompact in $M$ for $i=0,\dots,n,j=0,\dots,m$. Then $B=\prod_{i=0,j=0}^{n,m}B_{ij}$ be precompact in $(M^{(n+1)(m+1)},\rho'_{p})$.
\end{lemma}
\begin{proof}
The proof for $p\in(0,1)$ are similar to the case $p\in(1,\infty )$. Hence, we only give the proof for $p\in(1,\infty )$. Thus, we only consuder $p\in(1,\infty )$. It is obvious that $\rho'_{p}(\xi,\delta)\geq 0.$ If $\rho'_{p}(\xi,\delta)=0$, then $d(\xi_{00},\delta_{00})=0$, thus $\xi_{00}=\delta_{00}$. Let $i=1$. Then $d(\xi_{10}+\delta_{00},\xi_{00}+\delta_{10})= 0$, thus $d(\xi_{10},\delta_{10})=0$, i.e. $\xi_{10}=\delta_{10}$. Let $j=1$. Then it is similarly to get $\xi_{01}=\delta_{01}$. By induction we obtain $\xi_{i~j}=\delta_{i~j}$ for $i=0,\dots,n, j=0,\dots,m$. That means $\xi=\delta$. Let $h=(h_{00},\dots,h_{nm})\in M^{(n+1)(m+1)}$. By (1) and Minkowski's inequality, we have
\begin{align*}
&d(\xi_{00},\delta_{00})+\bigg(\sum_{i=1}^{n}d^{p}(\xi_{i\,0}+\delta_{i-1\,0},\xi_{i-1\,0}+\delta_{i\,0})\bigg)^{\frac{1}{p}}+\bigg(\sum_{j=1}^{m}d^{p}(\xi_{0\,j}+\delta_{0\,j-1},\xi_{0\,j-1}+\delta_{0\,j})\bigg)^{\frac{1}{p}}\\
&+\bigg(\sum_{j=1}^{m}\sum_{i=1}^{n}d^{p}(\xi_{i\,j}+\xi_{i-1\,j-1}+\delta_{i-1\,j}+\delta_{i\,j-1},\xi_{i-1\,j}+\xi_{i\,j-1}+\delta_{i\,j}+\delta_{i-1\,j-1})\bigg)^{\frac{1}{p}}\\
&\leq d(\xi_{00}+h_{00},\delta_{00}+h_{00})+\bigg(\sum_{i=1}^{n}d^{p}(\xi_{i\,0}+\delta_{i-1\,0}+h_{i-1\,0}+h_{i\,0},\xi_{i-1\,0}+\delta_{i\,0}+h_{i-1\,0}+h_{i\,0})\bigg)^{\frac{1}{p}}\\
&+\bigg( \sum_{j=1}^{m}d^{p}(\xi_{0\,j}+\delta_{0\,j-1}+h_{0\,j-1}+h_{0\,j},\xi_{0\,j-1}+\delta_{0\,j}+h_{0\,j-1}+h_{0\,j})\bigg)^{\frac{1}{p}}\\
&+\bigg(\sum_{j=1}^{m}\sum_{i=1}^{n}d^{p}(\xi_{i\,j}+\xi_{i-1\,j-1}+\delta_{i-1\,j}+\delta_{i\,j-1}+h_{i\,j}+h_{i-1\,j-1}+h_{i-1\,j}+h_{i\,j-1},\\
&~~~~~~~~~\xi_{i-1\,j}+\xi_{i\,j-1}+\delta_{i\,j}+\delta_{i-1\,j-1}+h_{i\,j}+h_{i-1\,j-1}+h_{i-1\,j}+h_{i\,j-1})\bigg)^{\frac{1}{p}}\\
&\leq d(\xi_{00},h_{00})+d(h_{00},\delta_{00})+\bigg(\sum_{i=1}^{n}d^{p}(\xi_{i\,0}+h_{i-1\,0},\xi_{i-1\,0}+h_{i\,0})\bigg)^{\frac{1}{p}}\\
&~~+\bigg(\sum_{i=1}^{n}d^{p}(h_{i\,0}+\delta_{i-1\,0},h_{i-1\,0}+\delta_{i\,0})\bigg)^{\frac{1}{p}}+\bigg(\sum_{j=1}^{m}d(\xi_{0\,j}+h_{0\,j-1},\xi_{0\,j-1}+h_{0\,j})\bigg)^{\frac{1}{p}}\\
&~~+\bigg(\sum_{j=1}^{m}d^{p}(h_{0\,j}+\delta_{0\,j-1},h_{0\,j-1}+\delta_{0\,j})\bigg)^{\frac{1}{p}}\\
&~~+\bigg(\sum_{j=1}^{m}\sum_{i=1}^{n}d^{p}(\xi_{i\,j}+\xi_{i-1\,j-1}+h_{i-1\,j}+h_{i\,j-1},\xi_{i-1\,j}+\xi_{i\,j-1}+h_{i\,j}+h_{i-1\,j-1})\bigg)^{\frac{1}{p}}\\
&~~+\bigg(\sum_{j=1}^{m}\sum_{i=1}^{n}d^{p}(h_{i\,j}+h_{i-1\,j-1}+\delta_{i-1\,j}+\delta_{i\,j-1},h_{i-1\,j}+h_{i\,j-1}+\delta_{i\,j}+\delta_{i-1\,j-1})\bigg)^{\frac{1}{p}}.
 \end{align*}
 Therefore, $\rho'_{p}(\xi,\delta)\leq\rho'_{p}(\xi,h)+\rho'_{p}(h,\delta)$. Thus $\rho'_{p}$ is a metic.

Let $\{\xi^{q}\}$ be a Cauchy sequence in $(M^{(n+1)(m+1)},\rho'_{p})$. For every $\varepsilon>0$, there exists a integer $N$ such that for $q>N$ and each $w\in \mathbb{N}$, $\rho'_{p}(\xi^{q},\xi^{q+w})<\varepsilon$, then $d(\xi_{00}^{q},\xi_{00}^{q+w})\leq\rho'_{p}(\xi^{q},\xi^{q+w})<\varepsilon$. Hence $\{\xi_{00}^{q}\}$ is a Cauchy sequence in $M$. By induction we get that $\{\xi_{ij}^{q}\}$ is a Cauchy sequence for every $i=0,1,\dots,n, j=0,1,\dots,m$. For every $i=0,1,\dots,n, j=0,1,\dots,m$, we denote $\lim_{q\to\infty}\xi_{ij}^{q}=\xi_{ij}$. If $q>N$, then
$\rho'_{p}(\xi^{q},\xi^{q+w})\leq \varepsilon$.
Let $w\to\infty$. Then $\rho'_{p}(\xi^{q},\xi)\leq\varepsilon$,
i.e. $\xi^{q}\to\xi$. Hence $M^{(n+1)(m+1)}$ is complete.

Suppose $B_{ij}$ is precompact for $i=0,1,\dots,n, j=0,1,\dots,m$. Let $\{\xi^{q}\}\in B_{00}\times \dots\times B_{nm}$, and $\{\xi_{ij}^{q}\}\in B_{ij}$,~$i=0,1,2,\dots,n, j=0,1,2,\dots,m$. There exists a subsequence $\{\xi_{00}^{q_{k_{00}}}\}$ which converges to $\xi_{00}$. By $\{\xi_{10}^{q_{k_{00}}}\}\subset B_{10}$, we know that exists a subsequence $\{\xi_{10}^{q_{k_{10}}}\}$ of $\{\xi_{10}^{q_{k_{00}}}\}$ such that $k_{10}\to\infty$, $\xi_{10}^{q_{k_{10}}}\to\xi_{10}$. By induction, there exists a sequence $\{\xi^{q_{k_{nm}}}\}$ such that $\lim_{k_{nm}\to\infty}\xi_{ij}^{q_{k_{nm}}}=\xi_{ij}$ for $i=0,1,\dots,n, j=0,1,\dots,m$. That means $\{\xi^{q_{k_{nm}}}\}$ is convergent. Therefore, $B$ is precompact in $M^{(n+1)(m+1)}$.
\end{proof}

\begin{theorem}\label{t2}
 Let $p\in(0,\infty )$. A set $A\subset {\rm WBV}_{p}(I\times I,M)$ is precompact if the following conditions are satisfied:

 {\rm(i)} the set $A$ is joint equivariated;

 {\rm(ii)} for every $t,s\in I$, the set $\{f(t,s):~f\in A\}$ is a precompact set of $M$.
\end{theorem}
\begin{proof}
 By condition {\rm(i)}, A is $\{\rho_{\Pi\times\Pi^\ast}:\Pi ~and~ \Pi^\ast ~are~ partitions~ of~ \rm I\}$-equimetric. For partitions $\Pi$ and $\Pi^{\ast}$, let us denote $B_{ij}=\{f(t_{i},s_{j}):~f\in A\}$,~$i=0,1,2,\dots,n$,~$j=0,1,2,\dots,m$. We know from condition {\rm(ii)} that $B_{ij}$ is a precompact set of $M$. Let $B=\prod_{j=0}^{m}\prod_{i=0}^{n}B_{ij}$. We know from Lemma \ref{l4} that $B$ is a totally bounded set in $M^{(n+1)(m+1)}$. Then for every $f\in A$, we define $Tf=\{f(t_i,s_j)\}^{n, m}_{i=0,j=0}\in M^{(n+1)(m+1)}$. From the definition, it is easy to see that $T$ is an isometry from $(A, \rho_{\Pi\times\Pi^{\ast}}) $ into $B$. Therefore $A$ is totally bounded in the pseudo-metric $\rho_{\Pi\times\Pi^\ast}$. By Lemma \ref{l1}, $A$ is an totally bounded set in $\rm{BV}(I\times I,M)$. Thus the proof is finished.
\end{proof}

\section{Bivariate bounded Riesz variation spaces}
The univariate bounded Riesz variation spaces were appeared in \cite{ri}. In this section we consider bivariate metric semigroup valued bounded Riesz variation spaces. Let $p\in(1,\infty )$. For a function $f:I\times I\rightarrow M$, a pair of partitions $\Pi$ and $\Pi^{\ast}$ of the unit interval $I$, we denote
\[
V_{p}^{R}(f(\cdot,0),\Pi):= \bigg(\sum_{i=1}^{n}\frac{d^{p}(f(t_{i},0),f(t_{i-1},0))}{(t_{i}-t_{i-1})^{p-1}}\bigg)^{\frac{1}{p}},
\]
\[
V_{p}^{R}(f(0,\cdot),\Pi^{\ast}):=\bigg(\sum_{j=1}^{m}\frac{d^{p}(f(0,s_{j}),f(0,s_{j-1}))}{(s_{j}-s_{j-1})^{p-1}}\bigg)^{\frac{1}{p}},
\]
\[
V_{p,2}^{R}(f,\Pi\times\Pi^{\ast}):=\bigg(\sum_{j=1}^{m}\sum_{i=1}^{n}\frac{d^{p}(f(t_{i},s_{j})+f(t_{i-1},s_{j-1}),f(t_{i},s_{j-1})+f(t_{i-1},s_{j}))}{(t_{i}-t_{i-1})^{p-1}(s_{j}-s_{j-1})^{p-1}}\bigg)^{\frac{1}{p}}.
\]
Moreover, $$V_{p}^{R}(f(\cdot,0)):= \sup V_{p}^{R}(f(\cdot,0),\Pi),$$ $$V_{p}^{R}(f(0,\cdot)):= \sup V_{p}^{R}(f(0,\cdot),\Pi^{\ast}),$$ $$V_{p,2}^{R}(f):= \sup V_{p,2}^{R}(f,\Pi\times\Pi^{\ast}),$$ where all suprema are taken over the indicated partitions.

\begin{definition}\label{d8}
Let $p\in(1,\infty )$. Given a function $f:~I\times I\to M$ if $$V_{p}^{R}(f)= V_{p}^{R}(f(\cdot,0))+V_{p}^{R}(f(0,\cdot))+V_{p,2}^{R}(f)<+\infty$$ then we say $f$ has bounded Riesz variation on $I\times I$ and denote ${\rm RBV}_{p}(I\times I,M)=\{f:I\times I\rightarrow M,V_{p}^{R}(f)<\infty\}$.
\end{definition}

Taking a pair of partitions $\Pi$ and $\Pi^{\ast}$ of the unit interval $I$, for every $f,g\in{\rm RBV}_{p}(I \times I,M)$, we denote
$$V_{p}^{R}(f(\cdot,0),g(\cdot,0),\Pi):= \bigg(\sum_{i=1}^{n}\frac{d^{p}(f(t_{i},0)+g(t_{i-1},0),f(t_{i-1},0)+g(t_{i},0))}{(t_{i}-t_{i-1})^{p-1}}\bigg)^{\frac{1}{p}},$$
$$V_{p}^{R}(f(0,\cdot),g(0,\cdot),\Pi^{\ast}):=\bigg(\sum_{j=1}^{m}\frac{d^{p}(f(0,s_{j})+g(0,s_{j-1}),f(0,s_{j-1})+g(0,s_{j}))}{(s_{j}-s_{j-1})^{p-1}}\bigg)^{\frac{1}{p}},$$ and
\begin{align*}
V_{p,2}^{R}(&f,g,\Pi\times\Pi^{\ast}):=\bigg(\sum_{j=1}^{m}\sum_{i=1}^{n}\frac{1}{(t_{i}-t_{i-1})^{p-1}(s_{j}-s_{j-1})^{p-1}}d^{p}(f(t_{i},s_{j})+f(t_{i-1},s_{j-1})\\
&+g(t_{i},s_{j-1})+g(t_{i-1},s_{j}),g(t_{i},s_{j})+g(t_{i-1},s_{j-1})+f(t_{i},s_{j-1})+f(t_{i-1},s_{j}))\bigg)^{\frac{1}{p}}.
\end{align*}

\begin{definition}\label{d9}
 Let $p\in(1,\infty )$. For every $f,g\in{\rm RBV}_{p}(I \times I,M)$ and each pair of partitions $\Pi$ and $\Pi^{\ast}$, we denote
\begin{align*}
V^{pR}_{\Pi\times\Pi^{\ast}}(f,g)=V_{p}^{R}(f(\cdot,0),g(\cdot,0),\Pi)+V_{R}(f(0,\cdot),g(0,\cdot),\Pi^{\ast})+V_{p,2}^{R}(f,g,\Pi\times\Pi^{\ast}).
\end{align*}
Then the joint variation of two functions $f$, $g$ is defined by $$V_{p}^{R}(f,g):= \sup_{\Pi\times\Pi^\ast} V^{pR}_{\Pi\times\Pi^\ast}(f,g)$$
 where the supremum is taken over the indicated partitions, and
\begin{align*}\rho_{R}(f,g):= \sup_{\Pi\times\Pi^\ast}\rho^{R}_{\Pi\times\Pi^\ast}(f,g)&:=\sup_{\Pi\times\Pi^\ast} \big(d(f(0,0),g(0,0))+V^{pR}_{\Pi\times\Pi^\ast}(f,g)\big)\\
&= d(f(0,0),g(0,0))+V_{p}^{R}(f,g) .\end{align*}
\end{definition}

Let $f,g\in{\rm RBV}_{p}(I\times I,M)$. By (\ref{2}) and Mikowski's inequality, we have $V_{p}^{R}(f,g)\leq V_{p}^{R}(f)+V_{p}^{R}(g)$. Thus $\rho_{R}(f,g)<\infty$, so the Definition \ref{d9} is reasonable.

\begin{lemma}\label{l6}
Let $p\in(1,\infty )$. If $f,g\in{\rm RBV}_{p}(I\times I,M)$, then $V_{p}^{R}(f)\leq V_{p}^{R}(f,g)+V_{p}^{R}(g)$ and $|V_{p}^{R}(f)-V_{p}^{R}(g)|\leq V_{p}^{R}(f,g)$.
\end{lemma}
\begin{proof}
Let $f,g\in{\rm RBV}_{p}(I\times I,M)$. For every partitions of $\Pi$ and $\Pi^{\ast}$. By (\ref{1}),
\begin{align*}
d(f(t_{i},0),f(t_{i-1},0))\leq d(f(t_{i},0)+g(t_{i-1},0),f(t_{i-1},0)+g(t_{i},0))+d(g(t_{i},0),g(t_{i-1},0)),
\end{align*}
\begin{align*}
d(f(0,s_{j}),f(0,s_{j-1}))\leq d(f(0,s_{j})+g(0,s_{j-1}),f(0,s_{j-1})+g(0,s_{j}))+d(g(0,s_{j}),g(0,s_{j-1})),
\end{align*}
and
\begin{align*}
&d(f(t_{i},s_{j})+f(t_{i-1},s_{j-1}),f(t_{i},s_{j-1})+f(t_{i-1},s_{j}))\\
&\leq d(f(t_{i},s_{j})+f(t_{i-1},s_{j-1})+g(t_{i},s_{j-1})+g(t_{i-1},s_{j}),g(t_{i},s_{j})+g(t_{i-1},s_{j-1})+f(t_{i},s_{j-1})\\
&+f(t_{i-1},s_{j}))+d(g(t_{i},s_{j})+g(t_{i-1},s_{j-1})+g(t_{i},s_{j-1})+g(t_{i-1},s_{j})).
\end{align*}
 By Minkowski's inequality, we have
\begin{align*}
V_{p}^{R}(f(\cdot,0),\Pi)&=\bigg(\sum_{i=1}^{n}\frac{d^{p}(f(t_{i},0),f(t_{i-1},0))}{(t_{i}-t_{i-1})^{p-1}}\bigg)^{\frac{1}{p}}\\
&\leq V_{p}^{R}(f(\cdot,0),g(\cdot,0),\Pi)+V_{p}(g(\cdot,0),\Pi),
\end{align*}
\begin{align*}
V_{p}^{R}(f(0,\cdot),\Pi^{\ast})&=\bigg(\sum_{j=1}^{m}\frac{d^{p}(f(0,s_{j}),f(0,s_{j-1}))}{(s_{j}-s_{j-1})^{p-1}}\bigg)^{\frac{1}{p}}\\
&\leq V_{p}^{R}(f(0,\cdot),g(0,\cdot),\Pi^{\ast})+V_{p}^{R}(g(0,\cdot),\Pi^{\ast}),
\end{align*}
and
\begin{align*}
V_{p,2}^{R}(f,\Pi\times\Pi^{\ast})&=\bigg(\sum_{j=1}^{m}\sum_{i=1}^{n}\frac{d^{p}(f(t_{i},s_{j})+f(t_{i-1},s_{j-1}),f(t_{i},s_{j-1})+f(t_{i-1},s_{j}))}{(t_{i}-t_{i-1})^{p-1}(s_{j}-s_{j-1})^{p-1}}\bigg)^{\frac{1}{p}}\\
&\leq V_{p,2}^{R}(f,g,\Pi\times\Pi^{\ast})+V_{p,2}^{R}(g,\Pi\times\Pi^{\ast}).
\end{align*}
Taking suprema over $\Pi$ and $\Pi^{\ast}$ on the last three inequalities, we have
\begin{align*}
V_{p}^{R}(f(\cdot,0))\leq V_{p}^{R}(f(\cdot,0),g(\cdot,0))+V_{p}^{R}(g(\cdot,0)),
\end{align*}
\begin{align*}
V_{p}^{R}(f(0,\cdot))\leq V_{p}^{R}(f(0,\cdot),g(0,\cdot))+V_{p}^{R}(g(0,\cdot)),
\end{align*}
and
\begin{align*}
V_{p,2}^{R}(f)\leq V_{p,2}^{R}(f,g)+V_{p,2}^{R}(g).
\end{align*}
Therefore,
\begin{align*}
V_{p}^{R}(f)\leq V_{p}^{R}(f,g)+V_{p}^{R}(g).
\end{align*}
Similarly, we have
\begin{align*}
V_{p}^{R}(g)\leq V_{p}^{R}(f,g)+V_{p}^{R}(f).
\end{align*}
Therefore, $|V_{p}^{R}(f)-V_{p}^{R}(g)|\leq V_{p}^{R}(f,g)$.
\end{proof}

\begin{theorem}\label{t5}
Let $p\in(1,\infty )$. Then $({\rm RBV}_{p}(I\times I,M),\rho_{R})$ is a complete metric space.
\begin{proof}  $\rho_{R}(f,g)\geq 0$ is evident. If $\rho_{R}(f,g)=0$, then $d(f(0,0),g(0,0))+V_{p}^{R}(f,g)=0$. Thus $f(0,0)=g(0,0)$. For every $t,s>0$, taking special partitions partitions $0=t_{0}<t<t_{1}=1$, and $0=s_{0}<s<s_{1}=1$, we have
 \begin{align*}
 &\bigg(\dfrac{d^{p}(f(t,0)+g(0,0),f(0,0)+g(t,0))}{(t-0)^{p-1}}+\dfrac{d^{p}(f(1,0)+g(t,0),f(t,0)+g(1,0))}{(1-t)^{p-1}}\bigg)^{\frac{1}{p}}\\
  &+\bigg(\dfrac{d^{p}(f(0,s)+g(0,0),f(0,0)+g(0,s))}{(s-0)^{p-1}}+\dfrac{d^{p}(f(0,1)+g(0,s),f(0,s)+g(0,1))}{(1-s)^{p-1}}\bigg)^{\frac{1}{p}}\\
  &+\bigg(\dfrac{d^{p}(f(t,s)+f(0,0)+g(t,0)+g(0,s),f(t,0)+f(0,s)+g(t,s)+g(0,0))}{(t-0)^{p-1}(s-0)^{p-1}}\\
  &+\dfrac{d^{p}(f(1,1)+f(t,s)+g(t,1)+g(1,s),f(t,1)+f(1,s)+g(t,s)+g(1,1))}{(1-t)^{p-1}(1-s)^{p-1}}\bigg)^{\frac{1}{p}}=0.
   \end{align*}
Hence
\begin{equation*}
\left\{
\begin{aligned}
&d^{p}(f(t,0)+g(0,0),f(0,0)+g(t,0))=0,\\
&d^{p}(f(0,s)+g(0,0),f(0,0)+g(0,s))=0,\\
&d^{p}(f(t,s)+f(0,0)+g(t,0)+g(0,s),f(t,0)+f(0,s)+g(t,s)+g(0,0))=0.\\
\end{aligned}
\right .
\end{equation*}
Therefore $f(t,0)=g(t,0)$,~$f(0,s)=g(0,s)$,~$f(t,s)=g(t,s)$ for every $t,s\in I$.

Again,
\begin{align*}
 V_{p}^{R}(f(\cdot,0),g(\cdot,0),\Pi)&=\bigg(\sum_{i=1}^{n}\frac{1}{(t_{i}-t_{i-1})^{p-1}}d^{p}(f(t_{i},0)+g(t_{i-1},0)+h(t_{i},0)+h(t_{i-1},0),\\
 &~~~~~~~~f(t_{i-1},0)+g(t_{i},0)+h(t_{i},0)+h(t_{i-1},0))\bigg)^{\frac{1}{p}}\\
  &\leq\bigg(\sum_{i=1}^{n}\frac{1}{(t_{i}-t_{i-1})^{p-1}}(d(f(t_{i},0)+h(t_{i-1},0),f(t_{i-1},0)+h(t_{i},0))\\
  &~~+d(g(t_{i-1},0)+h(t_{i},0),g(t_{i},0)+h(t_{i-1},0)))^{p}\bigg)^{\frac{1}{p}}\\
  &\leq\bigg(\sum_{i=1}^{n}\frac{d^{p}(f(t_{i},0)+h(t_{i-1},0),f(t_{i-1},0)+h(t_{i},0))}{(t_{i}-t_{i-1})^{p-1}}\bigg)^{\frac{1}{p}}\\
  &~~+\bigg(\sum_{i=1}^{n}\frac{d^{p}(g(t_{i-1},0)+h(t_{i},0),g(t_{i},0)+h(t_{i-1},0))}{(t_{i}-t_{i-1})^{p-1}}\bigg)^{\frac{1}{p}}\\
  &=V_{p}^{R}(f(\cdot,0),h(\cdot,0),\Pi)+V_{p}^{R}(h(\cdot,0),g(\cdot,0),\Pi).
 \end{align*}
Similarly, $V_{p}^{R}(f(0,\cdot),g(0,\cdot),\Pi^{\ast})\leq V_{p}^{R}(f(0,\cdot),h(0,\cdot),\Pi^{\ast})+V_{p}^{R}(h(0,\cdot),g(0,\cdot),\Pi^{\ast})$ and $V_{p,2}^{R}(f,g,\Pi\times\Pi^{\ast})\leq V_{p,2}^{R}(f,h,\Pi\times\Pi^{\ast})+V_{p,2}^{R}(h,g,\Pi\times\Pi^{\ast})$. On the other hand $d(f(0,0),g(0,0))\leq d(f(0,0),h(0,0))+d(h(0,0),g(0,0))$, so $\rho_{R}(f,g)\leq\rho_{R}(f,h)+\rho_{R}(h,g)$. Therefore $\rho_{R}$ is a metric of ${\rm RBV}_{p}(I\times I,M)$.

Now we prove ${\rm RBV}_{p}(I\times I,M)$ is complete. Let $\{f_{u}\}_u$ be a \rm Cauchy sequence of ${\rm RBV}_{p}(I\times I,M)$. Then for $\varepsilon>0$, there exists integer $N\in\mathbb{N}$, such that $\rho_{R}(f_u,f_v)\leq\varepsilon$ for $u,v>N$. For every $u,v>N$, we have $d(f_u(0,0),f_v(0,0))\leq\varepsilon$ and $V_{p}^{R}(f_u,f_v)\leq\varepsilon$. By the completeness of $M$, there exists $f(0,0)$ such that $\{f_{u}(0,0)\}_{u}$ converges uniformly to $f(0,0)$. Fixing an arbitrary point $t,s\in(0,1]$, takes special partitions $\{0,t,1\}$,~$\{0,s,1\}$, for every $u,v\in N$,
 \begin{align*}
 &(d^{p}(f_{u}(t,s),f_{v}(t,s)))^{\frac{1}{p}}-(d^{p}(f_{u}(0,0),f_{v}(0,0)))^{\frac{1}{p}}\\
 &\leq(d^{p}(f_{u}(t,s)+f_{v}(0,0),f_{u}(0,0)+f_{v}(t,s)))^{\frac{1}{p}}\\
 &\leq V_{p}^{R}(f_{u},f_{v})\leq\varepsilon.
  \end{align*}
Then $(d^{p}(f_{u}(t,s),f_{v}(t,s)))^{1/p}\leq2\varepsilon$ for every $u,v\in N$. It means that $\{f_{u}(t,s)\}_u$ is a \rm Cauchy sequence of $M$, then exists $f(t,s)$, such that $\{f_{u}(t,s)\}_{u}$ converges uniformly to $f(t,s)$. Thus $\{f_{u}\}_u$ converges uniformly to $f$ on $I\times I$.

By Lemma \ref{l6}, we have $V_{p}^{R}(f_{u})\leq C$ is bounded on ${\rm RBV}_{p}(I\times I,M)$. Let $V_{p}^{R}(f_{u})\leq C$. By Definition \ref{d9}, we have $V_{p}^{R}(f_{u}(\cdot,0))+V_{p}^{R}(f_{u}(0,\cdot))+V_{p,2}^{R}(f_{u})\leq C$. Let $u\to\infty$, then $V_{p}^{R}(f)\leq C$. Moreover, $V_{p}^{R}(f_{u},f_{v})\leq \varepsilon$ for $u,v>N$, then $V_{p}^{R}(f_{u}(\cdot,0),f_{v}(\cdot,0))+V_{p}^{R}(f_{u}(0,\cdot),f_{v}(0,\cdot))+V_{p,2}^{R}(f_{u},f_{v})\leq\varepsilon$. Let $v\to\infty$, then $V_{p}^{R}(f_{u},f)\leq \varepsilon$. Thus $\rho_{R}(f_u,f)\leq 2\varepsilon$ for $u>N$. Therefore $({\rm RBV}_{p}(I\times I,M),\rho_{R})$ is a complete metric space.
\end{proof}\end{theorem}

Similar to Lemma \ref{l4}, we have the following lemma.
\begin{lemma}\label{l7}
Let $p\in(1,\infty )$ and $m,n\in \mathbb{N}$. For elements $\xi=(\xi_{ij})_{i=0,j=0}^{n,m}$ and $\delta=(\delta_{ij})_{i=0,j=0}^{n,m}\in M^{(n+1)(m+1)}$, let
\begin{align*}
\rho'_{R}(&\xi,\delta)=\sup_{\Pi\times\Pi^{\ast}}\bigg(d(\xi_{00},\delta_{00})+\bigg(\sum_{i=1}^{n}\frac{d^{p}(\xi_{i\,0}+\delta_{i-1\,0},\xi_{i-1\,0}+\delta_{i\,0})}{(t_{i}-t_{i-1})^{p-1}}\bigg)^{\frac{1}{p}}\\
&~~+\bigg(\sum_{j=1}^{m}\frac{d^{p}(\xi_{0\,j}+\delta_{0\,j-1},\xi_{0\,j-1}+\delta_{0\,j})}{(s_{j}-s_{j-1})^{p-1}}\bigg)^{\frac{1}{p}}\\
&~~+\bigg(\sum_{j=1}^{m}\sum_{i=1}^{n}\frac{d^{p}(\xi_{i\,j}+\xi_{i-1\,j-1}+\delta_{i-1\,j}+\delta_{i\,j-1},\xi_{i-1\,j}+\xi_{i\,j-1}+\delta_{i\,j}+\delta_{i-1\,j-1})}{(t_{i}-t_{i-1})^{p-1}(s_{j}-s_{j-1})^{p-1}}\bigg)^{\frac{1}{p}}\bigg).
\end{align*}
Then $(M^{(n+1)(m+1)},\rho'_{R})$ is complete. Moreover, let $B_{ij}$ be precompact in $M$ for $i=0,\dots,n,j=0,\dots,m$. Then $B=\prod_{i=0,j=0}^{n,m}B_{ij}$ be precompact in $(M^{(n+1)(m+1)},\rho'_{R})$.
\end{lemma}

\begin{definition}\label{d10}
Let $p\in(1,\infty )$. A set $A\subset({\rm RBV}_{p}(I\times I,M),\rho_{R})$ is said to be joint equivariated, if for every $\varepsilon>0$, there exist partitions $\Pi_{\varepsilon}$ and $\Pi_{\varepsilon}^{\ast}$ such that for every $f,g\in A$,
$$V_{p}^{R}(f,g)\leq\varepsilon+V_{p}^{R}(f,g,\Pi_{\varepsilon}\times\Pi_{\varepsilon}^{\ast}),$$ where $$V_{p}^{R}(f,g,\Pi_{\varepsilon}\times\Pi_{\varepsilon}^{\ast}):= V_{p}^{R}(f(\cdot,0),g(\cdot,0),\Pi_{\varepsilon})+V_{p}^{R}(f(0,\cdot),g(0,\cdot),\Pi_{\varepsilon}^{\ast})+V_{p,2}^{R}(f,g,\Pi_{\varepsilon}\times\Pi_{\varepsilon}^{\ast}).$$
\end{definition}

\begin{theorem}\label{t6}
Let $p\in(1,\infty)$. A set $A\subset{\rm RBV}_{p}(I\times I,M)$ is precompact, if the following conditions are satisfied:

 {\rm(i)} $A$ is joint equivariated;

 {\rm(ii)} for every $t,s\in I$, the set $\{f(t,s):~f\in A\}$ is precompact.
\end{theorem}
By Lemma \ref{l1}, Lemma \ref{l7} and the same argument in the proof of Theorem \ref{t2} we can prove Theorem \ref{t6}, we leave the detail here.

\section{Bivariate bounded Waterman variation spaces}
In 1972, Waterman introduced the definition of $\Lambda$ bounded variation and applied it to problems related to Fourier series in \cite{28}. In this section, we will consider bounded $\Lambda$ variation functions with values in metric semigroup.
\begin{definition}\label{d4}
Let $\Lambda=\{\lambda_{i,j}\}$ be a decreasing sequences for the first and second variables respectively, and where $\lambda_{i,j}$ is a positive number, such that $\lambda_{n,m}\to 0$ as $n,m\to\infty$ and $\sum_{i,j=1}^{\infty}\lambda_{i,j}=\infty$.
\end{definition}

Given a function $f:I\times I\rightarrow M$, take a pair of partitions $\Pi$ and $\Pi^{\ast}$ of the unit interval $I$, we denote
\[
V_{\Lambda}(f(\cdot,0),\Pi):= \sum_{i=1}^{n}\lambda_{i,1}(d(f(t_{i},0),f(t_{i-1},0))),
\]
\[
V_{\Lambda}(f(0,\cdot),\Pi^{\ast}):=\sum_{j=1}^{m}\lambda_{1,j}(d(f(0,s_{j}),f(0,s_{j-1}))),
\]
and
\[
V_{\Lambda,2}(f,\Pi\times\Pi^{\ast}):=\sum_{j=1}^{m}\sum_{i=1}^{n}\lambda_{i,j}(d(f(t_{i},s_{j})+f(t_{i-1},s_{j-1}),f(t_{i},s_{j-1})+f(t_{i-1},s_{j}))).
\]
Moreover, $$V_{\Lambda}(f(\cdot,0)):= \sup V_{\Lambda}(f(\cdot,0),\Pi),$$ $$V_{\Lambda}(f(0,\cdot)):= \sup V_{\Lambda}(f(0,\cdot),\Pi^{\ast}),$$ and $$V_{\Lambda,2}(f):= \sup V_{\Lambda,2}(f,\Pi\times\Pi^{\ast}),$$ where all suprema are taken over the indicated partitions.

\begin{definition}\label{d5}
Let $\Lambda=\{\lambda_{i,j}\}$ as in Definition \ref{d4}. For a bounded function $f:~I\times I\to M$. The quantity $$V_{\Lambda}(f)= V_{\Lambda}(f(\cdot,0))+V_{\Lambda}(f(0,\cdot))+V_{\Lambda,2}(f)$$ is called the total Waterman of $f$ on $I\times I$. If $V(f)<+\infty$, we say that $f$ has bounded Waterman variation on $I\times I$ and write $\Lambda{\rm BV}(I\times I,M)=\{f:I\times I\rightarrow M,V_{\Lambda}(f)<\infty\}$.
\end{definition}

If $f\in\Lambda{\rm BV}(I\times I,M)$, then $f(\cdot,s)\in\Lambda{\rm BV}(I,M)$ for all $s\in I$. Similarly, $f(t,\cdot)\in\Lambda{\rm BV}(I,M)$, $V_{\Lambda}(f(\cdot,s),\Pi)\leq V_{\Lambda}(f(\cdot,0),\Pi)+V_{\Lambda,2}(f,\Pi\times\Pi^{\ast})$, and $V_{\Lambda}(f(t,\cdot),\Pi^{\ast})\leq V_{\Lambda}(f(0,\cdot),\Pi^{\ast})+V_{\Lambda,2}(f,\Pi\times\Pi^{\ast})$ for all $t\in I$.

For a pair of partitions $\Pi$ and $\Pi^{\ast}$ of the unit interval $I$, and every $f,g\in\Lambda{\rm BV}(I\times I,M)$, we denote
$$V_{\Lambda}(f(\cdot,0),g(\cdot,0),\Pi):= \sum_{i=1}^{n}\lambda_{i,1}(d(f(t_{i},0)+g(t_{i-1},0),f(t_{i-1},0)+g(t_{i},0))),$$
$$V_{\Lambda}(f(0,\cdot),g(0,\cdot),\Pi^{\ast}):=\sum_{j=1}^{m}\lambda_{1,j}(d(f(0,s_{j})+g(0,s_{j-1}),f(0,s_{j-1})+g(0,s_{j}))),$$ and
\begin{align*}
V_{\Lambda,2}(f,g,\Pi\times\Pi^{\ast}):=\sum_{j=1}^{m}\sum_{i=1}^{n}\lambda_{i,j}(d(&f(t_{i},s_{j})+f(t_{i-1},s_{j-1})+g(t_{i},s_{j-1})+g(t_{i-1},s_{j}),\\
&g(t_{i},s_{j})+g(t_{i-1},s_{j-1})+f(t_{i},s_{j-1})+f(t_{i-1},s_{j}))).
\end{align*}

\begin{definition}\label{d6}
Let $\Lambda=\{\lambda_{i,j}\}$ as in Definition \ref{d4}. For every $f,g\in\Lambda{\rm BV}(I\times I,M)$ and each pair of partitions $\Pi$ and $\Pi^{\ast}$, we denote $$V^{\Lambda}_{\Pi\times\Pi^\ast}(f,g):=V_{\Lambda}(f(\cdot,0),g(\cdot,0),\Pi)+V_{\Lambda}(f(0,\cdot),g(0,\cdot),\Pi^{\ast})+V_{\Lambda,2}(f,g,\Pi\times\Pi^{\ast}).$$
Then the joint variation of two functions $f$, $g$ is defined by $$V_{\Lambda}(f,g):= \sup_{\Pi\times\Pi^\ast} V^{\Lambda}_{\Pi\times\Pi^\ast}(f,g)$$
 where the supremum is taken over the indicated partitions, and we denote
\begin{align*}
\rho_{\Lambda}(f,g):= \sup_{\Pi\times\Pi^\ast}\rho^{\Lambda}_{\Pi\times\Pi^\ast}(f,g)&=:\sup_{\Pi\times\Pi^\ast} (d(f(0,0),g(0,0))+V^{\Lambda}_{\Pi\times\Pi^\ast}(f,g))\\
&= d(f(0,0),g(0,0))+V_{\Lambda}(f,g) ,
\end{align*} where $V_{\Lambda}(f,g)$ the supremum is taken over the indicated partitions.
\end{definition}

\begin{theorem}\label{t3}
The bounded variation space $(\Lambda{\rm BV}(I\times I,M),\rho_{\Lambda})$ is a complete metric space.
\end{theorem}
Similar to Theorem \ref{t1}, we have the following result.

\begin{definition}\label{d7}
A set $A\subset\Lambda{\rm BV}(I\times I,M)$ is said to be joint equivariated, if for every $\varepsilon>0$, there exist partitions $\Pi_{\varepsilon}$ and $\Pi_{\varepsilon}^{\ast}$ such that for every $f,g\in A$, we have
$$V_{\Lambda}(f,g)\leq\varepsilon+V_{\Lambda}(f,g,\Pi_{\varepsilon}\times\Pi_{\varepsilon}^{\ast}),$$ where $$V_{\Lambda}(f,g,\Pi_{\varepsilon}\times\Pi_{\varepsilon}^{\ast}):= V_{\Lambda}(f(\cdot,0),g(\cdot,0),\Pi_{\varepsilon})+V_{\Lambda}(f(0,\cdot),g(0,\cdot),\Pi_{\varepsilon}^{\ast})+V_{\Lambda,2}(f,g,\Pi_{\varepsilon}\times\Pi_{\varepsilon}^{\ast}).$$
\end{definition}
Similar to Lemma \ref{l4}, we have the following Lemma.
\begin{lemma}\label{l5}
Let $m,n\in \mathbb{N}$. For elements $\xi=(\xi_{ij})_{i=0,j=0}^{n,m}$ and $\delta=(\delta_{ij})_{i=0,j=0}^{n,m}\in M^{(n+1)(m+1)}$, let
\begin{align*}
\rho'_{\Lambda}(&\xi,\delta)=\sup_{\Pi\times\Pi^{\ast}}\bigg( d(\xi_{00},\delta_{00})+\sum_{i=1}^{n}\lambda_{i1}(d(\xi_{i\,0}+\delta_{i-1\,0},\xi_{i-1\,0}+\delta_{i\,0}))\\
&~~~~+\sum_{j=1}^{m}\lambda_{1j}(d(\xi_{0\,j}+\delta_{0\,j-1},\xi_{0\,j-1}+\delta_{0\,j}))\\
&+\sum_{j=1}^{m}\sum_{i=1}^{n}\lambda_{ij}(d(\xi_{i\,j}+\xi_{i-1\,j-1}+\delta_{i-1\,j}+\delta_{i\,j-1},\xi_{i-1\,j}+\xi_{i\,j-1}+\delta_{i\,j}+\delta_{i-1\,j-1}))\bigg).
\end{align*}
Then $(M^{(n+1)(m+1)},\rho'_{\Lambda})$ is complete. Moreover, let $B_{ij}$ be precompact in $M$ for $i=0,\dots,n,j=0,\dots,m$. Then $B=\prod_{i=0,j=0}^{n,m}B_{ij}$ be precompact in $(M^{(n+1)(m+1)},\rho'_{\Lambda})$.
\end{lemma}

By Lemmas \ref{l1} and \ref{l5}, and the same argument in the proof of Theorem \ref{t2} we have the following Theorem.
\begin{theorem}\label{t4}
A set~$A\subset\Lambda{\rm BV}(I\times I,M)$ is precompact, if the following conditions are satisfied:

 {\rm(i)} the set $A$ is joint equivariated;

 {\rm(ii)} for every $t,s\in I$, the set $\{f(t,s):~f\in A\}$ is precompact in $M$.
\end{theorem}

\section{Bivariate bounded Korenblum variation spaces}
In 1975, Boris Korenblum considered $\kappa$ variation in \cite{21}. In this section, we consider its analogue for metric semigroup valued functions with two variables.
\begin{definition}\label{d11}
A function $\kappa:[0,1]\to [0,1]$ is called a distortion function if $\kappa$ is increasing, concave, and satisfies $\kappa(0)=0$, $\kappa(1)=1$, and $\lim\limits_{t\to 0^{+}}\frac{\kappa(t)}{t}=\infty$.
\end{definition}

Obviously, a distortion function $\kappa$ is always subadditive in the sense that $\kappa(s+t)\leq\kappa(s)+\kappa(t)$ for every $0\leq s,t\leq 1$.

Let $p\in(1,\infty )$. For a function $f:I\times I\rightarrow M$, a pair of partitions $\Pi$ and $\Pi^{\ast}$ of the unit interval $I$, we denote
\[
V_{\kappa}(f(\cdot,0),\Pi):= \bigg(\sum_{i=1}^{n}\frac{d(f(t_{i},0),f(t_{i-1},0))}{\kappa(t_{i}-t_{i-1})}\bigg)^{\frac{1}{p}},
\]
\[
V_{\kappa}(f(0,\cdot),\Pi^{\ast}):=\bigg(\sum_{j=1}^{m}\frac{d(f(0,s_{j}),f(0,s_{j-1}))}{\kappa(s_{j}-s_{j-1})}\bigg)^{\frac{1}{p}}
\]
and
\[
V_{\kappa,2}(f,\Pi\times\Pi^{\ast}):=\bigg(\sum_{j=1}^{m}\sum_{i=1}^{n}\frac{d(f(t_{i},s_{j})+f(t_{i-1},s_{j-1}),f(t_{i},s_{j-1})+f(t_{i-1},s_{j}))}{\kappa(t_{i}-t_{i-1})\kappa(s_{j}-s_{j-1})}\bigg)^{\frac{1}{p}}.
\]
Moreover, $$V_{\kappa}(f(\cdot,0)):= \sup V_{\kappa}(f(\cdot,0),\Pi),$$ $$V_{\kappa}(f(0,\cdot)):= \sup V_{\kappa}(f(0,\cdot),\Pi^{\ast})$$ and $$V_{\kappa,2}(f):= \sup V_{\kappa,2}(f,\Pi\times\Pi^{\ast}),$$ where all suprema are taken over the indicated partitions.

\begin{definition}\label{d12}
Let a distortion function $\kappa:[0,1]\to [0,1]$. For a function $f:~I\times I\to M$, if $$V_{\kappa}(f):=  V_{\kappa}(f(\cdot,0))+V_{\kappa}(f(0,\cdot))+V_{\kappa,2}(f)<+\infty,$$ then we say that $f$ has bounded Korenblum variation on $I\times I$ and denote $\kappa{\rm BV}(I\times I,M)=\{f:I\times I\rightarrow M,V_{\kappa}(f)<\infty\}$.
\end{definition}

Taking a pair of partitions $\Pi$ and $\Pi^{\ast}$ of the unit interval $I$, for every $f,g\in\kappa{\rm BV}(I \times I,M)$, we denote
$$V_{\kappa}(f(\cdot,0),g(\cdot,0),\Pi):= \bigg(\sum_{i=1}^{n}\frac{d(f(t_{i},0)+g(t_{i-1},0),f(t_{i-1},0)+g(t_{i},0))}{\kappa(t_{i}-t_{i-1})}\bigg)^{\frac{1}{p}},$$
$$V_{\kappa}(f(0,\cdot),g(0,\cdot),\Pi^{\ast}):=\bigg(\sum_{j=1}^{m}\frac{d(f(0,s_{j})+g(0,s_{j-1}),f(0,s_{j-1})+g(0,s_{j}))}{\kappa(s_{j}-s_{j-1})}\bigg)^{\frac{1}{p}},$$ and
\begin{align*}
V_{\kappa,2}(&f,g,\Pi\times\Pi^{\ast}):=\bigg(\sum_{j=1}^{m}\sum_{i=1}^{n}\frac{1}{\kappa(t_{i}-t_{i-1})\kappa(s_{j}-s_{j-1})}d(f(t_{i},s_{j})+f(t_{i-1},s_{j-1})\\
&+g(t_{i},s_{j-1})+g(t_{i-1},s_{j}),g(t_{i},s_{j})+g(t_{i-1},s_{j-1})+f(t_{i},s_{j-1})+f(t_{i-1},s_{j}))\bigg)^{\frac{1}{p}}.
\end{align*}

\begin{definition}\label{d13}
 Let a distortion function $\kappa:[0,1]\to [0,1]$. For every $f,g\in\kappa{\rm BV}(I \times I,M)$ and each pair of partitions $\Pi$ and $\Pi^{\ast}$, we denote
\begin{align*}
V_{\Pi\times\Pi^{\ast}}^{\kappa}(f,g)=V_{\kappa}(f(\cdot,0),g(\cdot,0),\Pi)+V_{\kappa}(f(0,\cdot),g(0,\cdot),\Pi^{\ast})+V_{\kappa,2}(f,g,\Pi\times\Pi^{\ast}).
\end{align*}
Then the joint variation of two functions $f$, $g$ is defined by $$V_{\kappa}(f,g):= \sup_{\Pi\times\Pi^\ast} V^{\kappa}_{\Pi\times\Pi^\ast}(f,g)$$
 where the supremum is taken over the indicated partitions, and
\begin{align*}\rho_{\kappa}(f,g):= \sup_{\Pi\times\Pi^\ast}\rho^{\kappa}_{\Pi\times\Pi^\ast}(f,g)&=:\sup_{\Pi\times\Pi^\ast} (d(f(0,0),g(0,0))+V^{\kappa}_{\Pi\times\Pi^\ast}(f,g))\\
&= d(f(0,0),g(0,0))+V_{\kappa}(f,g) .\end{align*}
\end{definition}

Let $f,g\in \kappa{\rm BV}(I\times I,M)$. By (\ref{2}) and Mikowski's inequality, we have $V_{\kappa}(f,g)\leq V_{\kappa}(f)+V_{\kappa}(g)$. Thus $\rho_{\kappa}(f,g)<\infty$, so the Definition \ref{d13} is reasonable.

Similar to Theorem \ref{t1}, we have the following theorem.

\begin{theorem}\label{t7}
Let $\kappa$ be a distortion function on $[0,1]$. Then $(\kappa{\rm BV}(I\times I,M),\rho_{\kappa})$ is a complete metric space.
\end{theorem}

Similar to Lemma \ref{l4}, we have the following lemma.

\begin{lemma}\label{l8}
Let $\kappa$ be a distortion function and $m,n\in \mathbb{N}$. For elements $\xi=(\xi_{ij})_{i=0,j=0}^{n,m}$ and $\delta=(\delta_{ij})_{i=0,j=0}^{n,m}\in M^{(n+1)(m+1)}$, let
\begin{align*}
\rho'_{\kappa}&(\xi,\delta)=\sup_{\Pi\times\Pi^\ast}\bigg( d(\xi_{00},\delta_{00})+\bigg(\sum_{i=1}^{n}\frac{d(\xi_{i\,0}+\delta_{i-1\,0},\xi_{i-1\,0}+\delta_{i\,0})}{\kappa(t_{i}-t_{i-1})}\bigg)^{\frac{1}{p}}\\
&+\bigg(\sum_{j=1}^{m}\frac{d(\xi_{0\,j}+\delta_{0\,j-1},\xi_{0\,j-1}+\delta_{0\,j})}{\kappa(s_{j}-s_{j-1})}\bigg)^{\frac{1}{p}}\\
&+\bigg(\sum_{j=1}^{m}\sum_{i=1}^{n}\frac{d(\xi_{i\,j}+\xi_{i-1\,j-1}+\delta_{i-1\,j}+\delta_{i\,j-1},\xi_{i-1\,j}+\xi_{i\,j-1}+\delta_{i\,j}+\delta_{i-1\,j-1})}{\kappa(t_{i}-t_{i-1})\kappa(s_{j}-s_{j-1})}\bigg)^{\frac{1}{p}}\bigg).
\end{align*}
 Then $(M^{(n+1)(m+1)},\rho'_{\kappa})$ is complete. Moreover, let $B_{ij}$ be precompact in $M$ for $i=0,\dots,n,j=0,\dots,m$. Then $B=\prod_{i=0,j=0}^{n,m}B_{ij}$ is precompact in $(M^{(n+1)(m+1)},\rho'_{\kappa})$.
\end{lemma}

\begin{definition}\label{d14}
Let $\kappa$ be a distortion function on $[0,1]$. A set $A\subset\kappa{\rm BV}(I\times I,M)$ is said to be joint equivariated, if for every $\varepsilon>0$ exist a pair of partitions $\Pi_{\varepsilon}$ and $\Pi_{\varepsilon}^{\ast}$ such that for every~$f,g\in A$,
$$V_{\kappa}(f,g)\leq\varepsilon+V_{\kappa}(f,g,\Pi_{\varepsilon}\times\Pi_{\varepsilon}^{\ast}),$$ where $$V_{\kappa}(f,g,\Pi_{\varepsilon}\times\Pi_{\varepsilon}^{\ast}):= V_{\kappa}(f(\cdot,0),g(\cdot,0),\Pi_{\varepsilon})+V_{\kappa}(f(0,\cdot),g(0,\cdot),\Pi_{\varepsilon}^{\ast})+V_{\kappa,2}(f,g,\Pi_{\varepsilon}\times\Pi_{\varepsilon}^{\ast}).$$
\end{definition}

By Lemmas \ref{l1} and \ref{l8}, and the same argument in the proof of Theorem \ref{t2} we have the following Theorem.

\begin{theorem}\label{t8}
Let $\kappa$ be a distortion function. A set $A\subset\kappa{\rm BV}(I\times I,M)$ is precompact if the following conditions are satisfied:

{\rm(i)} $A$ is joint equivariated;

{\rm(ii)} for every $t,s\in I$, the set $\{f(t,s):~f\in A\}$ is precompact.
\end{theorem}

\vspace{0.5cm}

\noindent{\bf Acknowledgements} The work is supported by the National Natural Science Foundation of China (Grant No. 12161022).

\noindent{\bf  Data availability}  Our manuscript has no associated data.

\section*{Declarations}

 {\bf Conflict of interest} The authors declare that there is no conflict of interest.


\end{document}